\documentclass{amsproc}

\newtheorem{theorem}{Theorem}[section]
\newtheorem{corollary}[theorem]{Corollary}
\newtheorem{lemma}[theorem]{Lemma}
\newtheorem{proposition}[theorem]{Proposition}

\theoremstyle{definition}

\newtheorem{remark}[theorem]{Remark}

\numberwithin{equation}{section}

\def\lc{{\rm{lc}}}
\def\ame{{\rm{ame}}}
\def\sme{{\rm{sme}}}

\newcommand{\A}{\mathbb{A}}
\newcommand{\C}{\mathbb{C}}

\def\sC{{\mathcal{C}}}

\def\osigma{\overline{\sigma}}

\def\oB{\overline{B}}

\def\tC{{\tilde{C}}}

\def\P{{\mathbb{P}}}

\def\O{{\mathcal O}}

\def\dim{\mbox{dim}}

\def\oX{{\overline{X}}}
\def\oY{{\overline{Y}}}
\def\sX{{\mathcal{X}}}

\def\tC{{\widetilde{C}}}
\def\sM{\mathcal{M}}
\def\oM{\overline{M}}
\def\soM{\overline{\mathcal{M}}}

\begin{document}

\title[The solvable monodromy extension property and varieties of log general type]{The SME property and varieties of log general type}

\author{Sabin Cautis}
\address{Department of Mathematics, University of Southern California, Los Angeles, CA}
\email{cautis@usc.edu}
\thanks{The author is thankful for the support received through NSF grant DMS-1101439 and the Alfred P. Sloan foundation.}

\dedicatory{Dedicated to Joe Harris, with admiration on the occasion of his sixtieth birthday.}

\begin{abstract}
We speculate on the relationship between the solvable monodromy extension (SME) property and log canonical models. A motivating example is the moduli space of smooth curves which, by earlier work, is known to have this SME property. In this case the maximal SME compactification is the moduli space of stable nodal curves which coincides with its log canonical model. 
\end{abstract}

\maketitle

\tableofcontents

\section{Introduction}

On a late afternoon a few years ago, on the way back from one of his water cooler trips, Joe Harris dropped by my alcove with the following problem on his mind:
\vspace{.1in} \\
{\it Given a family of smooth curves over an open subscheme $U \subset S$, when does it extend to a family of stable curves over $S$? }
\vspace{.1in} \\
We had not been discussing ideas directly along these lines, so this question took me a little by surprise. It seemed like a very natural problem and it subsequently became a part of my thesis. Using the language of moduli spaces, it can be restated as follows:
\vspace{.1in} \\
{\it Given an open subscheme $U \subset S$ and a morphism $f: U \rightarrow \sM_{g,n}$ when does it extend to a regular morphism $S \rightarrow \soM_{g,n}$?}
\vspace{.1in} \\
Here $\sM_{g,n}$ denotes the moduli space of smooth genus $g$ curves with $n$ marked points and $\soM_{g,n}$ its Deligne-Mumford compactification. De Jong and Oort \cite{JO} show that if $D := S \setminus U$ is a normal crossing divisor then $f: U \rightarrow \sM_{g,n}$ extends to a regular morphism $S \rightarrow \soM_{g,n}$ assuming it does this over the generic points of $D$. Without this assumption one can still conclude that $f$ extends to a map $S \rightarrow \oM_{g,n}$ where $\oM_{g,n}$ is the coarse moduli space.

In general, lifting a map $S \rightarrow \oM_{g,n}$ to $\soM_{g,n}$ only requires taking a finite cover of $S$ ({\it i.e.} one does not need to blow up at all). From this point of view the second answer above suffices. When working over $\C$, it turns out to have the following generalization.

\begin{theorem}\cite[Thm. A]{C}\label{thmA:C}
Let $U \subset S$ be an open subvariety of an irreducible, normal variety $S$. A morphism $U \rightarrow \sM_{g,n}$ extends to a regular map $S \rightarrow \oM_{g,n}$ in a Zariski neighbourhood of $p \in S \setminus U$ if and only if the local monodromy around $p$ is virtually abelian.
\end{theorem}

\subsection{The abelian/solvable monodromy extension property}

Inspired by this result we introduced in \cite[Sect. 2]{C} the {\it abelian monodromy extension} (AME) property for a pair $(\sX, \oX)$ consisting of a Deligne-Mumford stack $\sX$ and a compactification $\oX$ of its coarse scheme. Roughly, $\oX$ is an AME compactification of $X$ if $U \rightarrow \sX$ extends to a regular map $S \rightarrow \oX$ whenever the image of the induced map $\pi_1(U) \rightarrow \pi_1(\sX)$ on fundamental groups is virtually abelian. Here $S$ is a small analytic neighbourhood of a point (so this is a local condition on the domain but a global condition on the target $(\sX, \oX)$). 

Among all AME compactifications of some $X$ there is a unique maximal one which we denote $X_{\ame}$ \cite[Cor. 3.7]{C}. This means that for any other AME compactification $\oX$ of $X$ there exists a birational morphism $X_{\ame} \rightarrow \oX$. For example, $\oM_{g,n}$ is the maximal AME compactification of $\sM_{g,n}$ \cite[Thm. 4.1]{C}. 

Although the AME condition is quite strong, there are abundant examples of pairs satisfying the AME property. More precisely:

\begin{proposition}\cite[Prop. 3.11]{C}\label{prop:lots}
Let $X \subset \oX$ be a dense, open immersion such that $\oX$ is a normal, complete variety. Then there exists an open $X^o \subset X$ such that $(X^o, \oX)$ has the AME property.
\end{proposition}

In this paper we will consider the (very similar) {\it solvable monodromy extension} (SME) property instead of the AME property. The definition is precisely the same except that we replace ``abelian'' with ``solvable'' everywhere. Once again, if $X$ has the SME property then there exists a maximal SME compactification $X_{\sme}$. In this case, since any abelian group is solvable, $X$ also has the AME property and there exists a regular morphism $X_{\ame} \rightarrow X_{\sme}$. 

\subsection{Relation to log canonical models}

If $C$ is a smooth curve, then it has the SME property if and only if it is stable (meaning that it has genus $g > 1$ or genus $g=1$ with at least one puncture or genus $g=0$ with at least three punctures). Notice how these are precisely the curves of log general type. 

We make the following two (completely wild) speculations:
\vspace{.2cm}\\
{\bf Speculation \#1.} {\it If $X$ has the SME property then its log canonical model $X_{\lc}$ (assuming it exists) is a compactification of $X$. In particular, $X$ is of log general type. } 
\vspace{.2cm}\\
{\bf Speculation \#2.} {\it If $X$ has the SME property then there exists a regular morphism $X_{\lc} \rightarrow X_{\sme}$ extending the identity map on $X$.} 
\vspace{.2cm}\\
\noindent In the case of surfaces we prove in Propositions \ref{prop:surfgentype} and \ref{prop:surfregmap} that this is indeed true. It should also be possible to verify these speculations if $X$ has a smooth log minimal model. 

For general higher dimensional varieties we describe a possible (but very sketchy) approach to proving these two claims. Along the way we bring up some related questions which may be of independent interest. 

The final section gives a simple example showing that varieties of log general type need not have the SME property. We also describe a potential application, explained to me by Sean Keel, to partial resolution of singularities. 

\subsection{Acknowledgments}

I would like to thank Maksym Fedorchuk, Brendan Hassett, Johan de Jong and Sean Keel for helpful, interesting discussions and the anonymous referee for finding a mistake and suggesting additional improvements and references. 

\section{Preliminaries}

We briefly discuss the two main concepts being related in this paper: the SME property and log canonical pairs. 

\subsection{The solvable monodromy extension (SME) property}

The solvable monodromy extension property is a direct analogue of the abelian monodromy extension property from \cite{C} (just take that definition and replace abelian with solvable). 

\subsubsection{Local monodromy}

We summarize the definition of local monodromy from \cite[Sec. 2.1]{C}. Let $X$ be an open subvariety of a normal variety $\oX$. Next, consider an open subvariety $U \subset S$ of a normal variety $S$ together with a morphism $U \rightarrow X$. Fix a connected, reduced, proper subscheme $T \subset S$. Now choose a sufficiently small analytic neighbourhood $V$ of $T$. We define the {\it local monodromy around $T$} as the image of fundamental groups
$$\mbox{Im} \left( \pi_1(V \cap U) \rightarrow \pi_1(X) \right).$$

Since $T$ is connected $V \cap U$ is connected and so the image of $\pi_1(V \cap U)$ is defined (up to conjugation) without having to choose a base point. Most commonly we will take $T$ to be a point $p \in S \setminus U$ to obtain the local monodromy around $p$. 

\subsubsection{The SME property}

Given an open embedding of normal varieties $X \subset \oX$, the pair $(X,\oX)$ has {\em the solvable monodromy extension (SME) property} if given any $U \subset S$ as above the morphism $U \rightarrow X$ extends to a regular map $S \rightarrow \oX$ in a neighbourhood of $p$ whenever the local monodromy around $p$ is virtually solvable (recall that a group is {\it virtually solvable} if it contains a solvable subgroup of finite index). 

In this case $\oX$ is complete and we say that $\oX$ is an {\em SME compactification} of $X$. We say $X$ has the SME property if it has an SME compactification. One can also define the SME property for stacks (see \cite[Sec. 2.2.1]{C}) but for simplicity we will only consider varieties. 

\subsubsection{Example: the moduli space of curves}

The main example from \cite{C} of a variety with the AME property is the moduli space of curves (Theorem \ref{thmA:C}). It turns out the moduli space of curves also has the SME property. 

\begin{corollary}
$\oM_{g,n}$ is the maximal SME compactification of $\sM_{g,n}$.
\end{corollary}
\begin{proof}
By \cite[Theorem B]{BLM} every solvable subgroup of $\pi_1(\sM_{g,n})$ is virtually abelian. This means that $\sM_{g,n}^{\ame} = \sM_{g,n}^{\sme}$ and the result follows from \cite[Thm. 4.1]{C} which states that $\oM_{g,n}$ is the maximal AME compactification of $\sM_{g,n}$. 
\end{proof}

\subsection{Log canonical pairs}

Consider a normal variety $X$ with a compactification $\oX$ so that the boundary $\Delta := \oX \setminus X$ is a normal crossing divisor. Note that in general one may have to first resolve $X$ in order to find such an $\oX$. 

We say that $X$ is of {\it log general type} if $(K_\oX + \Delta)$ is big. Then, assuming finite generation of the log canonical ring (say via the log minimal model program), the {\it log canonical model} of $X$ is 
$$X_{\lc} := \mbox{Proj} \left( \bigoplus_{m \ge 0} H^0(\O_{\oX}(m(K_\oX + \Delta)) \right).$$
If $X_{\lc}$ contains $X$ as an open subscheme then we say that $X_{\lc}$ is the {\it log canonical compactification of $X$}. Note that $X_{\lc}$ does not depend on the choice of $\oX$. By construction, assuming $X$ has (at worst) log canonical singularities, it follows that $X_{\lc}$ also has (at worst) log canonical singularities. 

In general, if $C,C'$ are two curves inside some proper variety $Y$, we will say that $C$ and $C'$ are numerically equivalent (denoted $C \sim C'$) if $C \cdot D = C' \cdot D$ for any Cartier divisor $D \subset Y$. 

\subsection{Some properties of AME compactifications} 

In \cite[Sect. 3]{C} we proved various basic results about AME compactifications. The key facts we used was that a subgroup of an abelian group is abelian and that the image of an abelian group is abelian. These facts also hold for solvable groups. Subsequently, the results from \cite[Sect. 3]{C} still hold if we replace ``AME'' with ``SME''. We now state three such results which we will subsequently use. 

\begin{lemma}\cite[Cor. 3.3]{C}\label{lem:SME1}
Suppose $(X,\oX)$ has the AME (resp. SME) property and let $Y \subset X$ be a closed, normal subscheme. If we denote by $\oY$ the normalization of the closure of $Y$ in $\oX$ then the pair $(Y,\oY)$ also has the AME (resp. SME) property.
\end{lemma}

\begin{lemma}\cite[Prop. 3.10]{C}\label{lem:SME2}
If $X$ has the AME (resp. SME) property and $i: Y \rightarrow X$ is a locally closed embedding then there exists a regular morphism $Y_{\sme} \rightarrow X_{\sme}$ (resp. $Y_{\ame} \rightarrow X_{\ame}$) which extends $i: Y \rightarrow X$.
\end{lemma}

\begin{lemma}\label{lem:SME3}
Suppose $X$ has the SME property and let $\oX$ be a compactification of $X$ equipped with a regular map $\pi: \oX \rightarrow X_{\sme}$. If $C \subset \oX$ is a curve so that the local monodromy around $C$ is virtually solvable then $\pi$ contracts $C$ to a point. 
\end{lemma}
\begin{remark}
In particular, this means that if $X$ has the SME property then its fundamental group is not virtually solvable. The analogous result for AME varieties also holds. 
\end{remark}
\begin{proof}
Choose a surface $S \subset \oX$, not contained in the boundary $\oX \setminus X$, but which contains $C$. Then the local monodromy around $C \subset S$ is virtually solvable and, by Lemma \ref{lem:SME1}, $S \cap X$ has the SME property. If we blow up $S$ then the proper transform $\tC$ of $C$ still has this property. Moreover, blowing up sufficiently we can assume that $\tC^2 < 0$.

Thus we end up with $\tC \subset S'$ where $\tC^2 < 0$ and a map $S' \rightarrow \oX$ which does not contract $\tC$. Moreover, the local monodromy around $\tC \subset S'$ is virtually solvable. But, since $\tC^2 < 0$, we can blow down $\tC$ to a point $q \in S''$. Then the local monodromy around $q$ is virtually solvable. This means that, in a neighbourhood of $q$, we get a regular map $S'' \rightarrow X_{\sme}$. Subsequently, the composition $S' \rightarrow \oX \xrightarrow{\pi} X_{\sme}$ contracts $\tC$ to a point. This means that $\pi$ contracts $C$ to a point.
\end{proof}

\section{A general property}

The following results, which we will apply later, serve as some indication that varieties satisfying the SME property are of log general type. 

\begin{proposition}\label{prop:main}
Let $X$ be a normal variety and $\oX$ some compactification such that $\Delta := \oX \setminus X$ is a divisor. If $C \subset \oX$ is a rational curve such that
\begin{itemize}
\item $C \cap X \ne \emptyset$ and $C \subset \oX_{\rm{smooth}}$, 
\item $(K_{\oX} + \Delta) \cdot C \le 0$ and $\Delta \cdot C \ge 3$
\end{itemize}
then there exists a curve $B_0$ and a map $F_0: \sC_0 := \P^1 \times B_0 \rightarrow \oX$ such that 
\begin{enumerate}
\item $F_0(\sC_0)$ is a surface, 
\item $F_0^{-1}(\Delta) = \{p_1, \dots, p_n\} \times B_0$ for some fixed points $p_1, \dots, p_n \in \P^1$,
\item there exists a section $\sigma$ of $\sC_0 \rightarrow B_0$ with $F_0(\sigma)=p$ for some $p \in X$. 
\end{enumerate}
\end{proposition}
\begin{proof}
We will use the following non-trivial deformation theory result \cite{M}. The deformation space of morphisms $f: \tC \rightarrow \oX$ has dimension at least
$$- (K_{\oX} \cdot C) + (1-g_\tC) \cdot \dim(X)$$
where $\tC$ is the normalization of $C$ and $g_\tC$ its genus. In our case $g_\tC=0$ and $-K_{\oX} \cdot C \ge \Delta \cdot C$ so the space of deformations has dimension at least $\dim(X) + \Delta \cdot C$. 

Now consider such a family of deformations $\pi: \sC \rightarrow B$ where $\dim(B) \ge \dim(X) + \Delta \cdot C$. Since the automorphism group of $\tC$ is 3-dimensional we can restrict $\sC$ to a smaller family $\sC' \rightarrow B' \subset B$ where $\dim(B') = \dim(B)-3$ and such that for any fibre in $\sC'$ there are only finitely many other fibres with the same image in $\oX$. Notice that here we need $\Delta \cdot C \ge 3$ or else this family may be empty. 

Suppose the general fibre of $\sC'$ intersects $\Delta$ in $n \ge 1$ distinct points. If $n \ge 3$ then, after possibly blowing up $B'$ and taking a finite cover, we obtain a map $B' \rightarrow \oM_{0,n}$ to moduli space of genus zero curves with $n$ marked points. Since $\dim(\oM_{0,n}) = n-3$ a generic fibre of this map has dimension at least 
$$\dim(B')-n+3 \ge \dim(X) + \Delta \cdot C - n \ge \dim(X).$$ 
The restriction of $\sC'$ to such a fibre leaves us with a family of curves $\pi'': \sC'' \rightarrow B''$ and a map $F: \sC'' \rightarrow \oX$. If $n \le 2$ then we take this to be the original family. 

Since $\dim(\sC'') = \dim(B'')+1 \ge \dim(\oX)+1$ the general fibre of $F$ has dimension at least one. Choose a general point $p \in F(\sC'') \cap X$ and let $B''' := \pi''(F^{-1}(p)) \subset B''$. Restricting $\sC''$ gives us a family $\pi''': \sC''' \rightarrow B'''$ with $\dim(B''') \ge 1$ with a map $F: \sC''' \rightarrow \oX$. After restricting even further we can assume for convenience that $\dim(B''')=1$.  

Now, by construction, there exists an open subset $B_0 \subset B'''$ such that the restriction of $\sC'''$ to it is isomorphic to $\sC_0 := \P^1 \times B_0$. We also have a map $F_0: \sC_0 \rightarrow \oX$ which takes $\cup_i p_i \times B_0$ to $\Delta$ and the rest to $X$ (here $p_1, \dots, p_n$ denote our $n$ marked points). Finally, after possibly pulling back to a finite cover of an open subset of $B_0$, there exists a section $\sigma$ of $\sC_0 \rightarrow B_0$ which is in the preimage of $p \in X$. 
\end{proof}

\begin{corollary}\label{cor:main}
Suppose $X$ is smooth, has the AME property and denote by $\oX$ a simple, normal crossing compactification of $X$. If $C \subset \oX$ is a rational curve not contained in the boundary $\Delta := \oX \setminus X$, then $(K_{\oX} + \Delta) \cdot C > 0$. 
\end{corollary}
\begin{proof}
By Lemma \ref{lem:SME3} we know $C \cdot \Delta \ge 3$. So suppose $(K_{\oX} + \Delta) \cdot C \le 0$ and consider the family $\sC_0 = (\P^1, p_1, \dots, p_n) \times B_0 \rightarrow B_0$ as in Proposition \ref{prop:main}. This family comes equipped with a map $F_0: \sC_0 \setminus \{\cup_i p_i \times B_0\} \rightarrow X$ and a section $\sigma$ such that $F_0(\sigma)=p \in X$. Compactify $B_0$ to some smooth curve $B$ and $\sC_0$ to the trivial product $\sC = (\P^1,p_1, \dots, p_n) \times B$. 

Since $\Delta$ is simple, normal crossing, we have a regular map $f: \oX \rightarrow X_{\ame}$ which extends the identity map on $X$. By the AME property, $f \circ F_0: \sC_0 \rightarrow X_{\ame}$ extends to a regular map $F: \sC \rightarrow X_\ame$. Then $F(\overline{\sigma})=p$ which means that $\osigma$ does not intersect any $p_i \times \oB_0$ (here $\osigma$ denotes the closure of $\sigma$). This means that the image of $\osigma$ under the projection from $\sC$ to $(\P^1,p_1, \dots, p_n)$ is a single point $q$. Thus $\osigma$ is just $\{q\} \times B$ which means $\osigma^2=0$. But this is impossible because $F: \sC \rightarrow X_\ame$ contracts $\osigma$ to a point.
\end{proof}

It would be interesting (and useful) to generalize Proposition \ref{prop:main} to curves $C$ of higher genus. In order to do that one needs a log version of the bend and break lemma (since then the analogue of Corollary \ref{cor:main} would almost say that $K_{\oX} + \Delta$ is nef). 
\vspace{.2cm} \\
{\bf Question 1.} {\it Is there a log version of the bend and break lemma? }

\section{The case of surfaces}

In this section suppose $X$ is a normal surface and denote by $\oX$ a simple normal crossing compactification of $X$. We denote by $\Delta \subset \oX$ the boundary $\Delta := \oX \setminus X$. In this case we have the following MMP for surfaces. 

\begin{theorem}\cite[Thm. 3.3, 8.1]{F}\label{thm:MMPsurfaces}
There exists a sequence of contractions 
$$(\oX,\Delta) = (\oX_0,\Delta_0) \xrightarrow{\phi_0} (\oX_1,\Delta_1) \xrightarrow{\phi_1} \dots \xrightarrow{\phi_{k-1}} (\oX_k,\Delta_k) = (X^*,\Delta^*)$$
such that each $(\oX_i,\Delta_i)$ is log-canonical and one of the following two things hold:
\begin{enumerate}
\item $K_{X^*} + \Delta^*$ is semi-ample or
\item there exists a morphism $g: X^* \rightarrow B$ such that $-(K_{X^*} + \Delta^*)$ is $g$-ample and $\dim(B) < 2$. 
\end{enumerate}
\end{theorem}
\begin{remark} The MMP states that in case (i) the divisor $K_{X^*} + \Delta^*$ is nef. Then by the abundance theorem for log-canonical surfaces, it is also semi-ample. 
\end{remark}

\begin{proposition}\label{prop:surfgentype}
If $X$ is a normal surface which has an AME compactification then $X$ is of log general type. 
\end{proposition}
\begin{proof}
Since we are trying to prove $X$ is of log general type we can remove the singular locus of $X$ and assume it is smooth. Now, compactify $X$ and consider a log minimal resolution $(\oX,\Delta)$. First we show that applying Theorem \ref{thm:MMPsurfaces} we end up in case (i). 

Suppose to the contrary that we are in case (ii). Denote the composition of contractions $\phi: (\oX,\Delta) \rightarrow (X^*, \Delta^*)$ and denote by $E_\phi$ the exceptional locus. There are two cases depending on whether $\dim(B)=0$ or $\dim(B)=1$. 

If $\dim(B)=0$ then $-(K_{X^*}+\Delta^*)$ is ample. By \cite[Cor. 1.6]{KeM}, $X^* \setminus \Delta^*$ is $\C^\times$-connected which means that $X$ is also $\C^\times$-connected. But then, by \cite[Cor. 7.9]{KeM}, we find that $\pi_1(X)$ is virtually abelian. This is impossible since $X$ has the AME property. See also \cite{Zh} for a similar approach.

If $\dim(B)=1$ then consider a connected component $C$ of a general fibre of $g$. Since $X^*$ is normal $C$ is smooth. The image $\phi(E_\phi) \subset X^*$ has codimension $2$ so $C \cap \phi(E_\phi) = \emptyset$. Hence $C$ has a unique lift to $\oX$ which we also denote $C$. Then 
\begin{equation}\label{eq:1}
(K_{\oX} + \Delta) \cdot C = (K_{X^*} + \Delta^*) \cdot C < 0
\end{equation}
because $-(K_{X^*} + \Delta^*)$ is $g$-ample. Since $C$ is a fibre of $g$ this means $C^2=0$ and hence 
$$\deg(K_C) + \Delta \cdot C = (K_\oX + C) \cdot C + \Delta \cdot C < 0.$$
Since $C$ is not contained in $\Delta$ we have $\Delta \cdot C \ge 0$. Thus $C$ has genus zero and $\Delta \cdot C \le 1$. So $C$ intersects $\Delta$ in at most one point, meaning that $\A^1 \subset C \setminus \Delta$. This is a contradiction since by Lemma \ref{lem:SME1}, $(C, \Delta \cap C) \subset (\oX, \Delta)$ is an AME pair.

So we must be in case (i). Consider the map $\pi: X^* \rightarrow B$ induced by $K_{X^*} + \Delta^*$. We must show that $\dim(B)=2$. If $\dim(B)=1$ then consider a general fibre $C$ of $\pi$ as above. Then 
$$(K_{\oX} + \Delta) \cdot C = (K_{X^*} + \Delta^*) \cdot C = 0$$
which is similar to equation (\ref{eq:1}). Then arguing as before, either $C$ has genus one and $\Delta \cdot C = 0$ or it has genus zero and $\Delta \cdot C \le 2$. In the first case $C$ does not intersect $\Delta$ and we get a contradiction by Lemma \ref{lem:SME1}. Similarly, in the second case $C$ intersects $\Delta$ in at most two points, meaning that $\C^\times \subset C \setminus \Delta$ which is again a contradiction by Lemma \ref{lem:SME1}. 

If $\dim(B)=0$ then $K_{X^*}+\Delta^*$ is trivial. We have 
$$K_{\oX} + \Delta = \phi^*(K_{X^*}+\Delta^*) + \sum_i a_i E_i$$
where $-1 \le a_i \le 0$ and $\{E_i\}$ are the irreducible, exceptional divisors. The left inequality follows since $(X^*,\Delta^*)$ is log canonical, while $a_i \le 0$ is a consequence of the fact $(\oX,\Delta)$ is a minimal log resolution (see \cite[Claim 66.3]{Ko1}). 

Thus, either all $a_i$ are zero and $\Delta = 0$ (which means $K_{\oX}$ is trivial) or some $a_i < 0$ (which means $\kappa(\oX) = -\infty$). The former cannot happen because the fundamental group of $X=\oX$ would be virtually abelian. In the second case, we can take any rational curve $C \subset \oX$ not contained in $\Delta$. Then $(K_\oX + \Delta) \cdot C = \sum_i a_i (E_i \cdot C) < 0$ which is a contradiction by Corollary \ref{cor:main}. 

Thus $\dim(B)=2$ and $X$ must be of log general type. 
\end{proof}

\begin{corollary}\label{cor:surfgentype}
Suppose $X$ is a normal surface with at worst log canonical singularities which has the AME property. Then  $X_{\lc}$ is a compactification of $X$.
\end{corollary}
\begin{proof}
Compactify $X$ to $\oX$ as in Theorem \ref{thm:MMPsurfaces}. By Proposition \ref{prop:surfgentype} $L^* := K_{X^*} + \Delta^*$ is semi-ample. It suffices to show that each $\phi_i$ as well as the map induced by $L^*$ only contracts curves in the boundary. 

Suppose $\phi_i$ contracts a curve $C$ which does not lie in the boundary. Then $(K_{\oX_i}+\Delta_i) \cdot C \le 0$ and $C$ is either rational or an elliptic curve which does not intersect $\Delta_i$ or $\mbox{sing}(\oX_i)$. The second case is not possible since the local monodromy around $C$ would be abelian (contradicting the fact that $X$ has the AME property). The first case is not possible by Corollary \ref{cor:main}.

The fact that the map induced by $L^*$ only contracts curves along the boundary follows similarly. 
\end{proof}

\begin{proposition}\label{prop:surfregmap}
Suppose $X$ is a normal surface with at worst log canonical singularities which has a maximal SME compactification $X_{\sme}$. Then there is a regular morphism $X_{\lc} \rightarrow X_{\sme}$ extending the identity map on $X$.
\end{proposition}
\begin{proof}
Denote by $(\oX,\Delta)$ a minimal simple normal crossing compactification of $X$ and denote by $(X_{\lc},\Delta_{\lc})$ its log canonical model. Then we have regular morphisms 
$$X_{\lc} \xleftarrow{\pi_1} \oX \xrightarrow{\pi_2} X_{\sme}.$$
Consider a connected, exceptional curve $E \subset \oX$ of $\pi_1$. If $\pi_1(E)$ is a point which does not intersect the boundary of $X_{\lc}$ then the local monodromy around $E$, is the same as the local monodromy around $\pi_1(E)$ which, by Proposition \ref{prop:solvable} is virtually solvable. Thus, by Lemma \ref{lem:SME3}, it follows that $\pi_2$ must contract $E$ to a point and hence $\pi_2$ factors through $\pi_1$.

If $\pi_1(E)$ intersect the boundary of $X_{\lc}$ then the type of singularity at $\pi_1(E)$ is very restricted (see \cite[Thm. 4.15]{KoM} for a list of possible singularities). Locally around $\pi_1(E)$, the complement of the boundary looks like the quotient of $\C^2 \setminus \{x=0\}$ or $\C^2 \setminus \{x=0,y=0\}$ by a finite group. Thus the monodromy is virtually abelian and, again as above, $\pi_2$ must contract $E$ to a point and hence $\pi_2$ factors through $\pi_1$.
\end{proof}

\begin{proposition}\label{prop:solvable}
Let $X$ be a normal surface and $x \in X$ a point. Then the following are equivalent:
\begin{enumerate}
\item $X$ has an at worst log canonical singularity at $x$,
\item the local fundamental group of $X$ around $x$ is solvable or finite.
\end{enumerate}
\end{proposition}
\begin{proof}
This follows by combining the results in \cite{K} and \cite{W}. 
\end{proof}

\section{The case of higher dimensional varieties}

\subsection{Generalizing Proposition \ref{prop:surfgentype}}\label{sec:gen1}

Suppose $X$ is a normal variety of arbitrary dimension which has an AME compactification. We would like to show that $X$ is of log general type by imitating the proof of Proposition \ref{prop:surfgentype}. For this purpose we can assume $X$ is smooth and we compactify it to some $\oX$ so that $\Delta := \oX \setminus X$ is a simple normal crossing divisor. 

The analogue of Theorem \ref{thm:MMPsurfaces} in this case is the (partially conjectural) log MMP. It states that there exists a sequence of birational maps 
$$(\oX,\Delta) = (\oX_0,\Delta_0) \overset{\phi_0}{\dashrightarrow} (\oX_1,\Delta_1) \overset{\phi_1}{\dashrightarrow} \dots \overset{\phi_{k-1}}{\dashrightarrow} (\oX_k,\Delta_k) = (X^*,\Delta^*)$$
such that each $(\oX_i,\Delta_i)$ is a log-canonical pair and each $\phi_i$ is either a divisorial contraction or a flip. As before, we denote the composition $\phi$ and note that $\phi^{-1}$ has no exceptional divisors. The resulting pair $(X^*,\Delta^*)$ satisfies one of the following two properties:
\begin{enumerate}
\item $K_{X^*} + \Delta^*$ is nef or
\item there exists a morphism $g: X^* \rightarrow B$ such that $-(K_{X^*} + \Delta^*)$ is $g$-ample and $\dim(B) < \dim(X)$. 
\end{enumerate}

First we would like to argue that we must be in case (i). Assume instead that we are in case (ii). If $\dim(B) > 0$ consider a general fibre $F$ and let $\Delta_F := \Delta^*|_F$. Then $K_F \cong K_{X^*}|_F$ and hence $K_F + \Delta_F = (K_{X^*} + \Delta^*)|_F$ is anti-ample. But the exceptional locus of $\phi^{-1}$ is codimension so there exists an open subset $F^o \subset F$ which is not of log general type but which sits inside $X$ and hence has an AME compactification. By induction on the dimension of $X$ this is impossible.

So we are left with considering the situation where $\dim(B)=0$. In this case $-(K_{X^*} + \Delta^*)$ is ample. When $X$ was a surface we used \cite{KeM}. Two analogues of such a result in higher dimensions were posed, for instance, during an open problem session at the AIM workshop ``Rational curves on algebraic varieties'' (May 2007):
\begin{itemize}
\item Is the fundamental group of the smooth locus of a log Fano variety finite?
\item Is a log Fano variety $(X^*,\Delta^*)$ log rationally connected? In other words, is there a rational curve passing through any two given points which intersects $\Delta^*$ only once? 
\end{itemize}
For us, an affirmative answer to a strictly easier question would suffice: 
\vspace{.2cm} \\
{\bf Question 2.} {\it If $(X^*,\Delta^*)$ is log Fano is there a rational curve which avoids any given locus $L$ of codimension at least $2$? }
\vspace{.2cm} \\
Taking $L$ to be the union of the exceptional locus of $\phi^{-1}$ and the singular locus we lift such a curve to $C \subset X$. Then, by Corollary \ref{cor:main}, we have 
$$(K_{X^*} + \Delta^*) \cdot C = (K_{\oX} + \Delta) \cdot C > 0$$
which contradicts the fact that $-(K_{X^*} + \Delta^*)$ is ample. 

So, assuming the answer to Question 2 is ``Yes'', we find that $K_{X^*} + \Delta^*$ is nef. The abundance conjecture, which was known in the case of surfaces (and also threefolds \cite{KMM}), would then imply that $K_{X^*}+\Delta^*$ is semi-ample. So we can consider the induced map $\pi: X^* \rightarrow B$. 

If $\dim(X) > \dim(B) > 0$ then choose a general fibre $F$. Then, proceeding as above, $K_F \cong K_{X^*}|_F$ and hence $K_F + \Delta_F$ is trivial (where $\Delta_F := \Delta^*|_F$). Thus, by induction on $\dim(X)$, this is a contradiction. 

If $\dim(B)=0$ then $K_{X^*}+\Delta^*$ is trivial. If $X^*$ were smooth then we would get a contradiction as follows. If $\Delta^* = 0$ then $X^*$ is Calabi-Yau and hence its fundamental group is virtually abelian. On the other hand, if $\Delta^* \ne 0$ then choose a point in $X^*$ which is not in the image of the exceptional locus of $\phi$. By the bend and break lemma we can find a rational curve $C$ through this point. Then, since $(K_{X^*} + \Delta^*) \cdot C < 0$, one can find a map $F_0$ and in Proposition \ref{prop:main}. Since $C$ does not belong to the image of the exceptional locus of $\phi$, this map can be lifted to a simple, normal crossing compactification of $X$ where we get a contradiction as in the proof of Corollary \ref{cor:main}. 

To push this argument through when $X^*$ is singular we could use an affirmative answer to the following question: 
\vspace{.2cm} \\
{\bf Question 3.} {\it Suppose $(X^*,\Delta^*)$ is log Calabi-Yau, meaning that $K_{X^*}+\Delta^*$ is trivial:
\begin{enumerate}
\item If $\Delta^* = 0$, is the fundamental group of the smooth locus of $X^*$ virtually abelian? 
\item If $\Delta^* \ne 0$, can you find a rational curve in $X^*$ which avoids the singular locus and is not contained in the image of the exceptional locus of $\phi$? 
\end{enumerate}}

Having ruled out the cases $\dim(X) > \dim(B)$ it follows that $\dim(X) = \dim(B)$ and $X$ is of log general type. 

\subsection{Generalizing Proposition \ref{prop:surfregmap}}\label{sec:gen2}

Let $(X_{\lc},\Delta_{\lc})$ be the log canonical model of $X$ and resolve it to a simple normal crossing compactification $(\oX,\Delta)$. Subsequently we have the following regular morphisms
$$X_{\lc} \xleftarrow{\pi_1} \oX \xrightarrow{\pi_2} X_{\sme}.$$
We would like to factor $\pi_1$ as a sequence of extremal divisorial contractions
$$(\oX,\Delta) \xrightarrow{\phi_0} (\oX_1,\Delta_1) \xrightarrow{\phi_1} \dots \xrightarrow{\phi_{k-1}} (\oX_k,\Delta_k) = (X_{\lc},\Delta_{\lc})$$
and to show by induction that:
\begin{enumerate}
\item $\pi_2$ descends to a regular map $\oX_i \xrightarrow{p_i} X_{\sme}$ and
\item any curve $C \subset \oX_i$ contracted by $\phi_i$ is also contracted by $p_i$. 
\end{enumerate}
Notice that (ii) implies that $p_i$ descends to a regular map $p_{i+1}: \oX_{i+1} \rightarrow X_{\sme}$ (so it suffices to prove (ii)). 

Unfortunately, it is not at all clear how to do this. One obvious approach is to hope that for a log canonical pair $(X,\Delta)$, the local fundamental group around a point $x \in X$ is virtually solvable. In the case of surfaces the answer to this question was ``yes'' but, as the referee points out, \cite[Thm. 2]{Ko2} gives a counter-example already in the case of 3-folds. In fact, \cite[Question 24]{Ko2} essentially asks if there are {\em any} natural restrictions on such local fundamental groups. 

\section{Some final remarks}

\subsection{Two examples}

Choose four points on $\P^2$ in general position and denote by $X$ the complement of the six lines through them. Then $X$ has the SME property with $X_{\sme}$ isomorphic to the blowup of $\P^2$ at these four points. In fact, $X \cong M_{0,5}$ parametrizing five points on $\P^1$ and this blowup is just $X_{\sme} \cong \oM_{0,5}$. 

On the other hand, choose $n \ge 4$ lines in general position on $\P^2$ and let $X$ be the complement of their union. Then by an old result of Zariski \cite{Za} the fundamental group of $X$ is abelian. Subsequently, $X$ does not have the SME (or AME) property even though it is of log general type (with log canonical model $X_{\lc} = \P^2$). 

\subsection{Resolution of singularities}

The following observation goes back to a conversation with Sean Keel a few years ago. 

Denote by $X(r,n)$ the moduli space of $n$ hyperplanes in $\P^{r-1}$. $X(r,n)$ has a particular compactification $\oX(r,n)$ known as Kapranov's Chow quotient. Keel and Tevelev \cite{KT} point out that the singularities of the boundary of $\oX(r,n)$ are arbitrarily complicated. 

On the other hand, one can show that $\oX(r,n)$ is an SME compactification. Hence, if Speculations \#1 and \#2 are correct, we obtain a regular map $X(r,n)_{\lc} \rightarrow \oX(r,n)$ which gives a canonical way to (partially) resolve arbitrarily bad singularities. 
\vspace{.2cm} \\
{\bf Question 4.} {\it Is there a regular map $X(r,n)_{\lc} \rightarrow \oX(r,n)$?}
\vspace{.2cm} \\

\bibliographystyle{amsalpha}

\end{document}